\documentclass[12pt]{amsart}
\usepackage{amscd,amssymb,epsfig,mathtools,setspace,subfig}%,accentset}
\usepackage[usenames,dvipsnames]{color}
\usepackage{url,times}
\usepackage{siunitx} %for number rounding

\usepackage{url}
\usepackage{enumitem} %enumerate package; see example below
\usepackage{float}%use [H] instead of [h] for tables (h is 'approximate' while H locks in place)

\calclayout

\def \[{ \mbox{[ \hspace{-.5em}[}} %for jumps, so \[v\]_e
\def \]{ \mbox{] \hspace{-.5em}]}}

\numberwithin{equation}{section} 

\newcommand{\ud}{\,d} 
 
\newcommand{\R}{\mathbb{R}}

\newcommand\T{{\mathcal T}}

\newcommand\cof{\operatorname{cof}}

\renewcommand{\div}{\operatorname{div}}

\newcommand{\tir}[1]{\ensuremath{\overline {#1}}} 
\newtheorem{thm}{Theorem}[section] 
 
\newtheorem{lemma}[thm]{Lemma}

\def\whsq{\vbox to 5.8pt 
{\offinterlineskip\hrule 
\hbox to 5.8pt{\vrule height 
5.1pt\hss\vrule height 5.1pt}\hrule}}

\def\<{\langle} 
\def\>{\rangle} 

\def\PP{{\mathop{{\rm I}\kern-.2em{\rm P}}\nolimits}} 
\def\FF{{\mathop{{\rm I}\kern-.2em{\rm F}}\nolimits}}   
\def\ZZ{{\mathop{{\rm I}\kern-.2em{\rm Z}}\nolimits}} 
 
\newlength{\sidemargin} 
\begin{document}
\title[]{
A two-grid method for the $C^0$
interior penalty discretization of the Monge-Amp\`{e}re equation
}

%\thanks{Not for dissemination, November 21, 2012. The authors thank Matthew Dornbos for useful discussions. The first author was supported in part by a 2009 Sloan Foundation Fellowship.  }

\thanks{Gerard Awanou was partially supported by NSF DMS grant \# 1720276 and Hengguang Li by NSF DMS grant \# 1418853 and by the Natural Science Foundation of China (NSFC) Grant 11628104}

\author{Gerard Awanou, Hengguang Li, and Eric Malitz}

\address{Department of Mathematics, Statistics, and Computer Science, M/C 249.
University of Illinois at Chicago, 
Chicago, IL 60607-7045, USA}
\email{awanou@uic.edu }  
%\urladdr{http://www.math.uic.edu/\~{}awanou}

\address{Department of Mathematics,
Wayne State University,
656 W. Kirby, Detroit, MI 48202, USA}
\email{hli@math.wayne.edu}
%\urladdr{http://www.math.wayne.edu/\~{}hli/}

\address{Department of Mathematics, Statistics, and Computer Science, M/C 249.
University of Illinois at Chicago, 
Chicago, IL 60607-7045, USA}
\email{emalit1@uic.edu }

\maketitle

\begin{abstract}
The purpose of this paper is to analyze an efficient method for the solution of the nonlinear system resulting from the discretization of the elliptic 
Monge-Amp\`ere equation by a $C^0$ interior penalty method with Lagrange finite elements. We consider the two-grid method for nonlinear equations which consists in solving the discrete nonlinear system on a coarse mesh and using that solution as initial guess for one iteration of Newton's method on a finer mesh. Thus both steps are inexpensive. We give quasi-optimal $W^{1,\infty}$ error estimates for the discretization and estimate the difference between the interior penalty solution and the two-grid numerical solution. Numerical experiments confirm the computational efficiency of the approach compared to Newton's method on the fine mesh.
\end{abstract}

\section{Introduction}
In this paper, we prove the convergence of a two grid method for solving the nonlinear system resulting from the discretization of the elliptic  Monge-Amp\`ere equation 
\begin{equation}
\det (D^2 u) = f \text{ in } \Omega, \quad u=g \text{ on } \partial \Omega,
\label{MA}
\end{equation}
with the  $C^0$ interior penalty discretization proposed in \cite{Brenner2010b}. The domain $\Omega$ is assumed to be a convex polygonal domain of $\R^2$ and \eqref{MA} is assumed to have a strictly convex smooth solution $u \in C^3(\tir{\Omega})$. The function $f \in C^1(\tir{\Omega})$ is given and satisfies $f \geq c_0$ for a constant $c_0 >0$  and the function $g\in C(\partial \Omega)$ is also given and
assumed to extend to a $C^3(\tir{\Omega})$ function $G$. In \eqref{MA}, $D^2 u=\big( \partial^2 u/(\partial x_i \partial x_j)\big)_{i,j=1,\ldots, 2} $ is the Hessian matrix of $u$ and $\det$ denotes the determinant operator. Let $V_h$ denote the Lagrange finite element space of degree $k \geq 2$ and let $D v$ denote the gradient of the function $v$. Recall that $\cof D^2 v$ denotes the matrix of cofactors of $D^2 v$. The $C^0$ interior penalty discretization can be written in abstract form as: find $u_h \in V_h$ such that $u_h=g_h$ on $\partial \Omega$ and
\begin{align} \label{IP-abstract}
A(u_h,\phi) & = 0, \forall \phi \in V_h \cap H_0^1(\Omega).
\end{align}
Here $g_h$ denotes the canonical interpolant in $V_h$ of a continuous extension of $g$  and $A$ is defined in \eqref{IP-form} below. The discretization has the property that if we denote by $A'(u;v, \phi)$ the Fr\'echet derivative evaluated at $u$ of the mapping $v \to A(v,\phi)$, then 
$A'(u;v,\phi)= \int_{\Omega} \big( (\cof  D^2 u) D v \big) \cdot D \phi \ud x,$ which gives the weak form of a standard linear elliptic operator. We exploit this property to give quasi-optimal $W^{1,\infty}$ error estimates, and the convergence of a two-grid numerical scheme for solving the discrete nonlinear system. Numerical experiments confirm the computational efficiency of the two-grid method compared to Newton's method on the fine mesh. %The analysis in this paper is based on the seminal work on two-grid methods in \cite{Xu1996}. 
Two-grid methods were initially analyzed in \cite{Xu1996}. The numerical results in \cite{Neilan2013} used a two-grid method.

Monge-Amp\`ere type equations with smooth solutions on polygonal domains appear in many problems of practical interest. For example they appear in the study of von K\'arm\'an model for plate buckling \cite{Brenner2017a}. In addition, for meteorological applications for which other differential operators are discretized with a finite element method, it would be advantageous to use a finite element discretization for the Monge-Amp\`ere operator as well. It is known that when $\Omega$ is strictly convex with a smooth boundary, and with our smoothness assumptions on $f$ and $g$, \eqref{MA} has a smooth solution. The analysis in this paper can be extended to such a framework by imposing weakly the boundary condition as in \cite{Brenner2010b}. There are several discretizations for smooth solutions of \eqref{MA}. Provably convergent schemes for non smooth solutions can be used for smooth solutions as well. However the latter have a low order of approximation for smooth solutions. We refer to \cite{Feng2013} for example for a review.  Because the interior penalty term involves the cofactor matrix of the Hessian, it is very likely that the method proposed in \cite{Brenner2010b} is suitable only for smooth solutions. It does not seem possible to put it in the framework of approximation by smooth solutions proposed in \cite{Awanou-Std04}, where the right hand side of \eqref{MA} is viewed as a measure. 

There has been no previous study of multilevel methods for finite element discretizations of \eqref{MA}. In addition, we
give quasi-optimal $W^{1,\infty}$ error estimates for the discretization with cubic and higher order elements. Localized  $W^{1,\infty}$ estimates were obtained in \cite{Neilan2013b} for quadratic and higher order elements on a smooth domain for which an elliptic regularity property holds for a linearization of the continuous problem. It is reasonable to assume that such a regularity property also holds for cubes. %\edit{ Justification? It is also reasonable to expect that $W^{1,\infty}$ estimates can be derived from the $H^1$ estimates obtained in \cite{Brenner2010b}. } Our derivation provides 
With the quasi-optimal $W^{1,\infty}$ error estimates we obtain a new proof of the optimal $H^1$ estimates obtained in \cite{Brenner2010b}.  %therein. 
 
 The paper is organized as follows. In the next section, we introduce additional notation and recall some preliminary results. 
 In section \ref{the-method} we %outline in a series of lemmas and theorems 
 give the $W^{1,\infty}$ error estimates for the discretization. %These results are proved in section \ref{proofs}. 
 In section \ref{two-grid-section} we present the two-grid algorithm and its error analysis. 
In section \ref{num}, we present numerical results which confirm the computational efficiency of the two-grid algorithm.

\section{Additional notation and Preliminaries}

Let $\mathcal{T}_h$ denote a conforming, shape regular and quasi-uniform triangulation of $\Omega$ into simplices $K$. We
denote by $\mathcal{E}_h$ the set of edges of $\mathcal{T}_h$, by $\mathcal{E}^b_h$ the set of boundary edges and by $\mathcal{E}^i_h$ the set of interior edges. 

Let $h_K$ denote the diameter of the element $K$ and put $h=\max_{K \in \mathcal{T}_h} h_K$. We assume that $0< h \leq 1$. We recall that, for a shape regular and quasi-uniform triangulation, there exists a constant $\sigma >0$ such that $h_K/\rho_K \leq \sigma$, for all $K \in \mathcal{T}_h$ where $\rho_K$ denotes the radius of the largest ball inside $K$ and there is another constant $C$ such that $h \leq C h_K$ for all $K \in \mathcal{T}_h$. Throughout the paper, we will use the letter $C$ for a generic constant, independent of $h$, which may change from occurrences.

We use the usual notation $W^{s,p}(\Omega)$, $1 \leq s, p \leq \infty$ for the Sobolev spaces of functions in $L^{p}(\Omega)$, $1 \leq p \leq \infty$, with weak derivatives up to order $s$ in $L^{p}(\Omega)$. The standard notation $H^s(\Omega)$ is used for $W^{s,2}(\Omega)$ and $H^{1}_{0}(\Omega)$ denotes the subspace of elements in $H^1(\Omega)$ with vanishing trace on the boundary $\partial \Omega$. Similarly, we define $W^{1,\infty}_{0}(\Omega)$. 
The norm of $v \in W^{k,p}(\Omega)$ is denoted $\|v\|_{W^{k,p}(\Omega)}$ and its seminorm $|v|_{W^{k,p}(\Omega)}$.
We will omit the argument $\Omega$ when it is understood from context. 

Denote by $P_k(K)$ the space of polynomials of degree $k$ on the element $K$. % and let $V_h$ denote the Lagrange finite element space of degree $k$. 
We will need the broken Sobolev norm defined for $1 \leq p < \infty$ by
\begin{equation*}
\|v\|_{W^{k,p}(\T_h)}=\left(\sum_{K \in \T_h}\|v\|_{W^{k,p}(K)}^p\right)^{1/p},
\end{equation*}
and for $p=\infty$ by
\begin{equation*}
\|v\|_{W^{k,\infty}(\T_h)}=\max_{K \in \mathcal{T}_h}\|v\|_{W^{k,\infty}(K)}.
\end{equation*}
We recall the inverse estimates \cite[Lemma 4.5.3]{Brenner08}
\begin{equation}\label{inverse}
\|v\|_{W^{s,p}(\mathcal{T}_h)}\leq Ch^{t-s+\min(0,\frac{2}{p}-\frac{2}{q})}\|v\|_{W^{t,q}(\mathcal{T}_h)}, \forall v\in V_h,
\end{equation}
valid for $0\leq t \leq s$ and $1\leq p,q \leq \infty$.  We also recall the trace inequality  \cite[Theorem 1.6.6]{Brenner08},
%\begin{equation}\label{trace-estimate}
$\|v\|_{L^p(\partial \Omega)}\leq C \|v\|_{W^{1,p}(\Omega)}, 1 \leq p \leq \infty$ which gives by a scaling argument
\begin{equation}\label{trace-estimate}
\|v\|_{L^p(\partial K)}\leq C h^{-\frac{1}{p}} (||v||_{L^p(K)} + h ||D v||_{L^p(K)} ).
\end{equation}
%\begin{equation}\label{trace}
%\|v\|_{L^p(\partial \Omega)}\leq C h^{-1/p} \|v\|_{L^p(\Omega)}, p \geq 2.
%\end{equation}
%from which it follows by an inverse estimate that
%\begin{equation}\label{trace-H2-to-H1}
%\|v\|_{H^2(\partial K)}\leq Ch^{-3/2}\|v\|_{H^1(K)}.
%\end{equation}
For $\phi \in W^{1,1}(\Omega)$, by the trace estimate \eqref{trace-estimate}, we have 
\begin{align} \label{w11-phi}
\sum_{e \in \mathcal{E}_h^i} \|\phi\|_{L^1(e)} \leq C h^{-1} \sum_{K \in \mathcal{T}_h}\|\phi\|_{W^{1,1}(K)} 
=C  h^{-1} \|\phi\|_{W^{1,1}}.
\end{align}
By an inverse estimate one has from \eqref{trace-estimate}
\begin{equation}\label{trace-inverse}
||v||_{L^2(\partial K)} \leq C h^{-\frac{1}{2}} ||v||_{L^2(K)} \, \forall v \in V_h.
\end{equation}
We will also need the following properties of the Lagrange interpolant operator $I_h$ \cite[Corollary 4.4.24]{Brenner08}
\begin{align} \label{nodal-interpolant}
||v-I_h v||_{W^{s,p}(\mathcal{T}_h)} & \leq C h^{k+1-s} ||v||_{W^{k+1,p}}, s=0, 1, 2 \text{ and } 1\leq p \leq \infty, v \in W^{k+1,p}.
%||v-I_h v||_{H^{1}} & \leq C h^k ||v||_{H^{k+1}}.
\end{align}
This follows from our assumptions on the triangulation and \cite[ (4.4.5) ]{Brenner08}, i.e. for $ v \in W^{k+1,p}(K)$
\begin{equation} \label{local-nodal-interpolant}
|v-I_h v|_{W^{s,p}(K)}  \leq C h^{k+1-s}_K |v|_{W^{k+1,p}(K)}, s=0, 1, 2 \text{ and } 1\leq p \leq \infty.
\end{equation}
It follows from \eqref{trace-estimate} and \eqref{local-nodal-interpolant} that 
%On the other hand, by a standard scaling argument and the Bramble-Hilbert lemma, 
\begin{equation} \label{bramble}
\begin{split}
|| D( I_h u - u) ||_{L^{2}(\partial K)} & \leq C h^{- \frac{1}{2} }  || D( I_h u - u) ||_{L^{2}(K)} + C h^{  \frac{1}{2} }  || D(I_h u - u) ||_{H^{1}(K)} \\
& \leq C h^{k- \frac{1}{2}} || u ||_{H^{k+1}(K)}. 
\end{split}
\end{equation}

For two matrices $A$ and $B$, $A:B = \sum^2_{i,j=1}A_{ij}B_{ij}$ denotes their Frobenius inner product.
% where $n$ is the total number of entries of each matrix. 
The divergence of a matrix field is understood as the vector obtained by taking the divergence of each row. 
%As an example,

The following results can be checked by simple algebraic computations and can also be found in \cite{AwanouLiMixed1}. For $v$ sufficiently smooth we have
\begin{equation}\label{det1}
\det D^2 v = \frac{1}{2} (\cof  D^2 v): D^2 v,
%=\frac{1}{2} \div\big((\cof  D^2 u) D u \big), 
\end{equation}
and if $F(v) = \det D^2 v -f$, the Fr\'echet derivative of $F$ at $v$ is given by  
\begin{equation}\label{det3} 
F'(v) w = (\cof  D^2 v): D^2 w,
\end{equation}
for $v,w$ sufficiently smooth. 
%Also, for $v, w$ sufficiently smooth,
Under the same assumptions
\begin{equation}\label{det2}
\div\big((\cof D^2 w) D v )= (\cof  D^2 w): D^2 v,
\end{equation}
which is a consequence of the product rule and the row divergence-free property of the Hessian, i.e. $\div( \cof  D^2 w) =0$.
% c.f. for example \cite{evans1998partial} p. 440.
It then follows that
\begin{equation}\label{det4} 
F'(u)v= \div \big((\cof  D^2 u) D v \big). 
\end{equation}
Note that 
\begin{equation} \label{linearity}
\cof(D^2 v +D^2 w) = \cof(D^2 v) + \cof(D^2 w),
\end{equation}
 since we restrict our discussion to the two dimensional case.  
We have \cite{Brenner2010b}
\begin{equation}\label{det5}
\det D^2 v - \det D^2 w = \frac{1}{2} (\cof(D^2 v) + \cof(D^2 w)): (D^2 v - D^2 w).
\end{equation}
Using \eqref{det2}, \eqref{linearity} and  \eqref{det5}, we obtain
\begin{equation}\label{det6}
\det D^2 v - \det D^2 w = \frac{1}{2}\div \Big( \big( \cof D^2 (v+w)\big) D(v -w)\Big).
\end{equation}
We recall that $u$ is strictly convex and thus $\cof D^2 u $ is uniformly positive definite; that is, there exists positive constants $\alpha_0$ and $\alpha_1$ such that $\forall x \in \mathbb{R}^2$
\begin{equation}\label{SPD}
\alpha_0 |r|^2 \leq r^T (\cof D^2 u(x) ) r \leq \alpha_1 |r|^2, \quad \forall r \in \mathbb{R}^2.  
\end{equation}
For $v \in V_h$, we will make the abuse of notation of denoting by $D^2v$ the discrete Hessian of $v$ computed element by element.

Next, we recall some algebraic manipulations of discontinuous functions. For $e \subset \partial K$, let $n_K$ denote the outward normal to $K$ and let $v|_K$ denote the restriction of the field $v$ to $K$. For $e=K^+ \cap K^-$, we define the jump of the vector field $v$ across $e$ as
\begin{equation}\label{jump}
\[v\]_e=v|_{K^+} \cdot n_{K^+} + v|_{K^-}  \cdot n_{K^-}, 
\end{equation}
and its average on $e$ as
\begin{equation}\label{average}
 \{\!\{v\}\!\}_e= \frac{1}{2}(v|_{K^+}  + v|_{K^-} ).
\end{equation}
The jump and average of a matrix field $E$ on $e$ are defined respectively as
\begin{equation}\label{matrix-jump}
\[E\]_e=n_{K^+}^T E|_{K^+}  +n_{K^-}^T E|_{K^-}, 
\end{equation}
and
\begin{equation}\label{matrix-average}
 \{\!\{E\}\!\}_e= \frac{1}{2}(E|_{K^+} + E|_{K^-}).
\end{equation}
For a matrix field $E$ and a vector field $v$ it is not difficult to check that for $e \in \mathcal{E}^i_h$
\begin{equation}\label{jumpaverage}
\[Ev\]_e=\[\{\!\{E\}\!\}_e v\]_e+\[E\]_e\{\!\{v\}\!\}_e. %no cdot: it's 1x2 normal times 2x2 A times 2 X1 average of v 
\end{equation}
We will omit below the subscript $e$ as it will be clear from the context.

Let $\hat{P}_h:H_0^1(\Omega) \rightarrow V_h$ denote the projection with respect to the bilinear form $A'(u;\cdot,\cdot)$ and recall that $G$ denotes a $C^3(\tir{\Omega})$ extension of $g$. We define 
$$
P_h u = \hat{P}_h(u-G) + I_h G,
$$
where $I_h$ denotes the canonical Lagrange interpolant operator into $V_h$.
Then
$P_h u = I_h u$ on $\partial \Omega$ and
\begin{equation}\label{2grid-proj-def}
A'(u;P_h u - u, \phi) = A'(u;I_h G - G, \phi) ,~ \forall \phi \in V_h \cap H_0^1(\Omega).
\end{equation}
Put $w=u-G$. Since $w=0$ on $\partial \Omega$ and $w$ is smooth, we have \cite[Corollary 8.1.12]{Brenner08}
\begin{equation*}
||w- \hat{P}_h(w)||_{W^{1,\infty}} \leq C h^{k} ||w||_{W^{k+1,\infty}}, \ \text{for } w \in W^{k+1,\infty}(\Omega), 
w=0 \text{ on } \partial \Omega. % \label{2grid-proj1}  
\end{equation*}
Therefore
\begin{align*}
||u-P_h u||_{W^{1,\infty}} & = ||w+ G - \hat{P}_h w-I_h G ||_{W^{1,\infty}} \\
& \leq ||w- \hat{P}_h w||_{W^{1,\infty}} +||G-I_h G||_{W^{1,\infty}}.
\end{align*}
It thus follows from the approximation properties of $I_h$ that
\begin{equation}
||u-P_h u||_{W^{1,\infty}} \leq C h^{k} ||u||_{W^{k+1,\infty}} \equiv C_1 h^k, \ \text{for } u \in W^{k+1,\infty}(\Omega).  \label{2grid-proj1}  
\end{equation}

By an inverse estimate, \eqref{2grid-proj1}, and the approximation properties of $I_h$, we have
\begin{align*}
||u-P_h u||_{W^{2,\infty}(\mathcal{T}_h)} &\leq ||u-I_h u ||_{W^{2,\infty}(\mathcal{T}_h)} + ||I_h u -P_h u||_{W^{2,\infty}(\mathcal{T}_h)} \\
& \leq C h^{k-1} ||u||_{W^{k+1,\infty}}  + C h^{-1} \big(  ||I_h u - u||_{W^{1,\infty}} +  || u-P_h u||_{W^{1,\infty}}  \big),
\end{align*}
that is
\begin{align} 
||u-P_h u||_{W^{2,\infty}(\mathcal{T}_h)} &\leq C h^{k-1} ||u||_{W^{k+1,\infty}} \label{2grid-proj2}. %\equiv C h^{k-1}.
\end{align}

%Note that since $\cof D^2 u$ is symmetric, for $v, w \in H^1_0(\Omega)$ we have $A(u;v,w)=A(u;w,v)$. 

Following \cite{Xu1996}, we will obtain pointwise estimates via the use of discrete Green's functions. For $z\in \Omega \setminus \cup_{K \in \mathcal{T}_h}\partial K$, let $g^z_{h,i} \in V_h \cap H_0^1(\Omega)$, $i=1,2$ be defined by:
\begin{equation}\label{2grid-green1}
A'(u; g^z_{h,i},\phi)=\frac{\partial \phi}{\partial x_i}(z),  \quad \forall \phi \in V_h \cap H_0^1(\Omega),
\end{equation}
and let $G^z_h$ be defined by
\begin{equation}\label{2grid-green2}
A'(u;G^z_h,\phi) = \phi (z), \quad \forall \phi \in V_h \cap H_0^1(\Omega).
\end{equation}
We have for $h$ sufficiently small
%\begin{equation}\label{2grid-L2greenbound}
%\|g^z_{h,i}\|_{L^2} \leq C |\ln h|^{1/2},
%\end{equation}
\begin{equation}\label{2grid-W11greenbound}
\|g^z_{h,i}\|_{W^{1,1}} \leq C |\ln h|,
\end{equation}
and
\begin{align*} 
\|G^z_h\|_{L^{2}} \leq C,
\end{align*}
where the constant $C$ is independent of $z$. In the case $\cof D^2 u$ is the identity matrix, the proof is given in \cite[Lemmas 2.1 and 3.3]{ThomeeXuZhang}. The proof of the general case is similar \cite{Malitz-thesis}. Moreover
we have \cite{Malitz-thesis}, see also \cite[Exercise 8.x.19]{Brenner08},
\begin{align} \label{green-0}
\|G^z_h\|_{W^{1,1}} \leq C|\ln h|.
\end{align}
%To prove that $|G^z_h|_{W^{1,1}} \leq C|\ln h|$, one uses H\"older's inequality and $A'(u;G^z_h,G^z_h)=G^z_h(z)$ combined with a discrete Sobolev inequality.  that
%$$
%|| G^z_h ||_{W^{1,1}}  \leq C || G^z_h ||_{W^{1,2}} \leq C | \ln h|^{\frac{1}{2}} \leq C | \ln h|. 
%$$
%We argue as in \cite[Exercise 8.x.19]{Brenner08}. We recall the
It is enough to prove that $|G^z_h |_{W^{1,1}} \leq C|\ln h|$ which follows from the bound
$| G^z_h |_{W^{1,2}} \leq C | \ln h|^{\frac{1}{2}} $ and H\"older's inequality. 
By the discrete Sobolev inequality \cite[(4.9.2)]{Brenner08} and the coercivity of the form $A'(u; .,.)$, we have for $h$ sufficiently small
\begin{align*}
| A'(u;G^z_h,G^z_h) | = | G^z_h (z)| \leq || G^z_h ||_{L^{\infty}} \leq C  | \ln h|^{\frac{1}{2}} || G^z_h ||_{H^{1}} \leq C  | \ln h|^{\frac{1}{2}} 
|A'(u;G^z_h,G^z_h)|^{\frac{1}{2}}. 
\end{align*}
It follows that
$$
| G^z_h |_{W^{1,2}}  \leq C |A'(u;G^z_h,G^z_h)|^{\frac{1}{2}} \leq C  | \ln h|^{\frac{1}{2}},
$$
giving the claimed bound.

\section{$W^{1,\infty}$ error estimates for the $C^0$ Interior penalty discretization} \label{the-method}

We first describe the interior penalty discretization proposed in \cite{Brenner2010b} for polygonal domains and with the boundary condition enforced strongly. For $\phi \in H_0^1(\Omega)$ and $v \in H^2(K)$ for all $K \in \mathcal{T}_h$,
we define
\begin{equation}\label{IP-form}
A(v,\phi) := \sum_{K \in \mathcal{T}_h} \int_{K} (f-\det D^2 v) \phi \ud x + \sum_{e \in \mathcal{E}_h^i } \int_e \[\{\!\{\cof D^2 v \}\!\} D v \] \phi
\ud s.
\end{equation}
Recall that the discrete problem is given by \eqref{IP-abstract}. The addition of the penalty terms, the second term on the right of \eqref{IP-form}, to the natural discretization of \eqref{MA} is motivated by the need in the analysis that the Fr\'echet derivative evaluated at $u$ of the mapping $v \to A(v,\phi)$ is given by
\begin{equation}\label{Frechet-at-u}
A'(u;v,\phi)= \int_{\Omega} \big( (\cof  D^2 u) D v \big) \cdot D \phi \ud x.
\end{equation} 
This is proven in \cite[p. 5]{Brenner2010b}. For the convenience of the reader, we give the proof %as Lemma \ref{interior-penalty-frechet-lemma} below.
in the next lemma.

Let $R(w;v,\phi)$ denote the remainder of the Taylor expansion at $w$ of $w \mapsto A(w,\phi)$, i.e. 
\begin{equation} \label{rdef0}
A(w+v,\phi)=A(w,\phi)+A'(w;v,\phi) + R(w;v,\phi).
\end{equation}

\begin{lemma}\label{interior-penalty-frechet-lemma}
For $v, w \in H^2(K)$ for all $K \in \mathcal{T}_h$ and $\phi \in H^1_0(\Omega)$, we have
\begin{align}\label{IP-frechet}
\begin{split}
A'(w;v,\phi) &  = \sum_{K \in \mathcal{T}_h} \int_{K }\big( (\cof  D^2 w)  D v \big) \cdot D \phi \ud x \\
& \quad - \sum_{e \in \mathcal{E}_h^i }
\int_e \[(\cof  D^2 w)\] \{\!\{D v\}\!\} \phi  \ud s 
 + \sum_{e \in \mathcal{E}_h^i } \int_e \[\{\!\{\cof D^2 v \}\!\} D w \] \phi
\ud s,
\end{split}
\end{align}
and
\begin{align} \label{rdef1}
R(w;v,\phi) = -\sum_{K \in \mathcal{T}_h}  \int_{K} (\det D^2v)\phi\ud x + \sum_{e \in \mathcal{E}_h^i } \int_e \[\{\!\{\cof D^2 v \}\!\} D v \] \phi\ud s.
\end{align}
In particular, for $u \in C^3(\Omega)$, \eqref{Frechet-at-u} holds.
%\begin{equation*}
%A'(u;v,\phi)=  \sum_{K \in \mathcal{T}_h}  \int_{K} \big((\cof  D^2 u) D v\big) \cdot D \phi \ud x.
%\end{equation*}

\end{lemma}

\begin{proof}
For $w,v \in H^2(K)$ for all $K \in \mathcal{T}_h$ we have
\begin{multline*}
A(w+v,\phi)-A(w,\phi)= -\sum_{K \in \mathcal{T}_h}  \int_{K} \big(\det D^2(w+v)-\det D^2w\big)\phi\ud x \\+ 
\sum_{e \in \mathcal{E}_h^i } \int_e \[\{\!\{\cof D^2 (w+v) \}\!\} D (w+v) \] \phi \ud s - \sum_{e \in \mathcal{E}_h^i } \int_e \[\{\!\{\cof D^2 w \}\!\} D w \] \phi
\ud s,
\end{multline*}
for all $\phi \in H^1_0(\Omega)$.
Since $D^2 w$ is a $2 \times 2$ matrix, $\det D^2(w+v)= \det D^2 w + \det D^2 v +\cof D^2 w : D^2 v$. 

Thus 
\begin{multline*}
A(w+v,\phi)-A(w,\phi)= -\sum_{K \in \mathcal{T}_h}  \int_{K} (\det D^2v)\phi\ud x  -\sum_{K \in \mathcal{T}_h}  \int_{K} (\cof D^2 w: D^2v)\phi\ud x \\+ \sum_{e \in \mathcal{E}_h^i } \int_e \[\{\!\{\cof D^2 w \}\!\} D v \] \phi\ud s + \sum_{e \in \mathcal{E}_h^i } \int_e \[\{\!\{\cof D^2 v \}\!\} D w \] \phi\ud s.
+ \sum_{e \in \mathcal{E}_h^i } \int_e \[\{\!\{\cof D^2 v \}\!\} D v \] \phi\ud s.
\end{multline*}
By \eqref{det2} $\cof  D^2 w: D^2 v = \div\big((\cof D^2 w) D v )$.
This implies that
\begin{align*} %\label{interior-penalty-frechet-non-ibp}
A'(w;v,\phi) & = -\sum_{K \in \mathcal{T}_h} \int_{K} \div \big((\cof  D^2 w) D v\big) \phi \ud x +
\sum_{e \in \mathcal{E}_h^i } \int_e \[\{\!\{\cof D^2 w \}\!\} D v \] \phi 
\ud s \\ & \qquad \qquad \quad + \sum_{e \in \mathcal{E}_h^i } \int_e \[\{\!\{\cof D^2 v \}\!\} D w \] \phi
\ud s,
\end{align*}
and $R(w;v,\phi)$ is given by \eqref{rdef1}.
%\begin{align*} 
%R(w;v,\phi) = -\sum_{K \in \mathcal{T}_h}  \int_{K} (\det D^2v)\phi\ud x + \sum_{e \in \mathcal{E}_h^i } \int_e \[\{\!\{\cof D^2 v \}\!\} D v \] \phi\ud s.
%\end{align*}
By integration by parts and using the fact $\phi=0$ on $\partial \Omega$,
\begin{align*}
A'(w;v,\phi) 
& = \sum_{K \in \mathcal{T}_h} \int_{K} \big((\cof  D^2 w)  D v \big) \cdot D \phi \ud x - \sum_{e \in \mathcal{E}_h^i }
\int_e \[(\cof  D^2 w) D v \] \phi   \ud x \\
& \qquad \qquad \quad  +
\sum_{e \in \mathcal{E}_h^i } \int_e \[\{\!\{\cof D^2 w \}\!\} D v \] \phi 
\ud s  + \sum_{e \in \mathcal{E}_h^i } \int_e \[\{\!\{\cof D^2 v \}\!\} D w \] \phi
\ud s .
\end{align*}
%Assume that $\phi$ is continuous, so that $\{\phi\}=\phi, [\phi]=0$. 
By \eqref{jumpaverage},
$
\[(\cof  D^2 w) D v\] = \[\{\!\{\cof  D^2 w \}\!\} D v \] + \[\cof  D^2 w\] \{\!\{D v\}\!\}.
$
It follows that

\begin{align*}
A'(w;v,\phi) &  = \sum_{K \in \mathcal{T}_h} \int_{K }\big( (\cof  D^2 w)  D v \big) \cdot D \phi \ud x - \sum_{e \in \mathcal{E}_h^i }
\int_e \[(\cof  D^2 w)\] \{\!\{D v\}\!\} \phi  \ud s \\\notag
& \qquad \qquad \quad  + \sum_{e \in \mathcal{E}_h^i } \int_e \[\{\!\{\cof D^2 v \}\!\} D w \] \phi
\ud s .
\end{align*}
%Given the regularity of $u$, we obtain by a density argument
Finally, since by assumption $u \in C^3(\Omega)$, $\[(\cof  D^2 u)\]=0$. In addition, by definition $\{\!\{\cof D^2 v \}\!\}$ is continuous and $Du$ is continuous by the assumption on $u$. Thus $\[\{\!\{\cof D^2 v \}\!\} D u \]=0$ on each interior edge. The statement about $A'(u;v,\phi)$ easily follows.
\end{proof}

\begin{lemma}
We have for $\phi \in H_0^{1}(\Omega)$ and $v, w \in H^2(K)$ for all $K \in \mathcal{T}_h$,
\begin{multline} \label{R1-simple}
%\begin{split}
R(w;v,\phi)  = \frac{1}{2} \sum_{K \in \mathcal{T}_h}  \int_K [(\cof D^2 v) D v ]  \cdot D \phi \ud x \\
-\frac{1}{2}  \sum_{e \in \mathcal{E}_h^i } \int_e \[(\cof D^2 v) \] \{\!\{ D v \}\!\} \phi \ud s 
 + \frac{1}{2} \sum_{e \in \mathcal{E}_h^i } \int_e \[ \{\!\{\cof D^2 v \}\!\} D v \] \phi\ud s.
%\end{split}
\end{multline}
\end{lemma}

\begin{proof}

Using \eqref{det1} and integration by parts
\begin{multline*}
 \int_{K} (\det D^2 v ) \phi \ud x = \frac{1}{2}  \int_{K} (\cof D^2 v ): (D^2 v) \, \phi \ud x = \frac{1}{2}  \int_{K} \div\big((\cof D^2 v) D v ) \phi \ud x
 \\
 =-  \frac{1}{2} \int_K [(\cof D^2 v) D v ] \cdot D \phi \ud x + \frac{1}{2}  \int_{\partial K} \big[ (\cof D^2 v) D v \big] \cdot n_K \phi \ud s.
\end{multline*}
By definition of jump and since $\phi=0$ on $\partial \Omega$, we have
\begin{align}
\sum_{K \in \mathcal{T}_h} \int_{\partial K} \big( (\cof D^2 v) D v \big) \cdot n_K \phi \ud s & = \sum_{e \in \mathcal{E}_h^i } \int_e \[(\cof D^2 v) D v \] \phi \ud s  \label{part-22}.
\end{align}
We conclude that
\begin{multline*}
R(w;v,\phi) = \frac{1}{2} \sum_{K \in \mathcal{T}_h}  \int_K [(\cof D^2 v) D v ] \cdot D \phi \ud x 
-\frac{1}{2}  \sum_{e \in \mathcal{E}_h^i } \int_e \[(\cof D^2 v) D v \] \phi \ud s \\
+ \sum_{e \in \mathcal{E}_h^i } \int_e \[\{\!\{\cof D^2 v \}\!\} D v \] \phi\ud s.
\end{multline*}
Therefore, using \eqref{jumpaverage} to expand the term $\[(\cof D^2 v) D v \] $ we obtain \eqref{R1-simple}.

%\begin{multline*}
%R(w;v,\phi) = \frac{1}{2} \sum_{K \in \mathcal{T}_h}  \int_K [(\cof D^2 v) D v ]  \cdot D \phi \ud x 
%-\frac{1}{2}  \sum_{e \in \mathcal{E}_h^i } \int_e \[(\cof D^2 v) \] \{\!\{ D v \}\!\} \phi \ud s \\
%+ \frac{1}{2} \sum_{e \in \mathcal{E}_h^i } \int_e \[ \{\!\{\cof D^2 v \}\!\} D v \] \phi\ud s.
%\end{multline*}

\end{proof}

We define a nonlinear operator $\Phi: V_h \to V_h$ by  $ v_h = \Phi( v_h)$ on $\partial \Omega$ and
\begin{equation} \label{Phi-def}
A'(u; v_h-\Phi( v_h),\phi) = A_{}( v_h,\phi), \quad \forall \phi \in V_h \cap H_0^1(\Omega).
\end{equation}
A fixed point of $\Phi$ is a solution of the nonlinear finite element problem \eqref{IP-abstract}. We note that by \eqref{nodal-interpolant},  \eqref{2grid-W11greenbound} and \eqref{green-0}
\begin{align*}
| A'(u;I_h G - G, G^z_h) | & \leq C_2 h^k | \ln h| \\
| A'(u;I_h G - G, g^z_{h,i}) | & \leq C_3 h^k | \ln h|.
\end{align*}
We then define $C_4= \max\{ C_1, C_2, C_3 \}$ where the constant $C_1$ is defined in \eqref{2grid-proj1}.
%For $0< \epsilon <1$ arbitrary, 
Consider the closed set
\begin{equation}\label{2grid-ball}
B_h = \{ v \in V_h: v=g_h \text{ on } \partial \Omega, ||v-u||_{W^{1,\infty}} \leq 3 C_4 h^{k} |\ln h |\}.
\end{equation}
 By  \eqref{2grid-proj1} $P_h u \in B_h$ and hence $B_h$ is non-empty. %But $u \not\in V_h$? 
%The proofs of the next two lemmas are given in section \ref{proofs}.
%\edit{Obviously, the size of the ball can't be arbitrary. Probably not also the choice of $C_1$}

\begin{lemma} \label{invariant}
We have $\Phi( B_h) \subset  B_h$ for $h$ sufficiently small and $k \geq 3$.
\end{lemma}
\begin{proof}
For $v_h \in B_h$, %$A(u;\Phi(v_h)-P_h u,\phi) = A(u;\Phi(v_h)- u,\phi)$.
we have using \eqref{Phi-def} and \eqref{2grid-proj-def}
\begin{align*}
\begin{split}
A'(u;\Phi( v_h)- P_h u,\phi) %& = A(u;\Phi(v_h)- u,\phi) \\
&= A'(u;\Phi( v_h)-  v_h,\phi)+A'(u; v_h- P_h u,\phi) \\
&= -A_{}( v_h,\phi) + A'(u; v_h-  u,\phi) +  A'(u; u - P_h u,\phi) \\
&= -A_{}( v_h,\phi) + A'(u; v_h-  u,\phi) - A'(u;I_h G - G, \phi). 
\end{split}
\end{align*}
By definition of the residual \eqref{rdef0}, and since $A(u,\phi)=0$, we have
$$
-A_{}( v_h,\phi) + A'(u; v_h-  u,\phi) = A(u,\phi) -A_{}( v_h,\phi) + A'(u; v_h-  u,\phi) = -R(u;v_h-u,\phi).
$$
We conclude that
\begin{align} \label{expr-inv}
A'(u;\Phi( v_h)- P_h u,\phi) = -R(u;v_h-u,\phi) - A'(u;I_h G - G, \phi).
\end{align}
Therefore, using \eqref{R1-simple}, \eqref{w11-phi} and \eqref{Frechet-at-u}, we have % and the definition of $C_4$, we have
\begin{multline*}
| A'(u;\Phi( v_h)- P_h u,\phi)   | \leq C || v_h-u ||_{W^{2,\infty}(\mathcal{T}_h)}  ||  v_h-u ||_{W^{1,\infty}}  || \phi ||_{W^{1,1}} \\ +C  ||  v_h-u ||_{W^{2,\infty}(\mathcal{T}_h)}  ||  v_h-u ||_{W^{1,\infty}} 
 \sum_{e \in \mathcal{E}_h^i } || \phi ||_{L^1(e)} +  |A'(u;I_h G - G, \phi) | \\
 \leq C h^{-1} || v_h-u ||_{W^{2,\infty}(\mathcal{T}_h)}  ||  v_h-u ||_{W^{1,\infty}}  || \phi ||_{W^{1,1}} + |A'(u;I_h G - G, \phi)|.
\end{multline*}
%By \eqref{2grid-proj1}, we have $||v_h-u||_{W^{1,\infty}} \leq ||v_h-P_h u||_{W^{1,\infty}} + ||P_h u-u||_{W^{1,\infty}} \leq 2 C_1 h^{k} |\ln h |$. 
By definition of $B_h$,  $||v_h-u||_{W^{1,\infty}} \leq C h^{k} |\ln h |$. 
Moreover, by triangle inequality, \eqref{2grid-proj2}
and an inverse estimate
$||v_h-u||_{W^{2,\infty}(\mathcal{T}_h)} \leq ||v_h-P_h u||_{W^{2,\infty}(\mathcal{T}_h)} + ||P_h u-u||_{W^{2,\infty}(\mathcal{T}_h)} \leq C h^{k-1} |\ln h | + C h^{k-1}  \leq C h^{k-1} |\ln h | $. Thus
\begin{multline*}
| A'(u;\Phi( v_h)- P_h u,\phi)   | \leq C h^{k-2}  |\ln h | \, || \phi ||_{W^{1,1}}   h^{k} |\ln h | + |A'(u;I_h G - G, \phi)| \\
\leq (C h^{k-2} |\ln h |^2  ) h^{k} || \phi ||_{W^{1,1}} + | A'(u;I_h G - G, \phi)| .
\end{multline*}
Taking $\phi=g^z_{h,i}$ with the estimate \eqref{2grid-W11greenbound}, and taking $\phi=G^z_h$ with the estimate \eqref{green-0}, we obtain using the definition of $C_4$
$$
\|\Phi( v_h)- P_h u\|_{W^{1,\infty}} \leq   (C h^{k-2} |\ln h |^2 + C_4 ) h^{k } |\ln h|.
$$
Since $C  h^{k-2} |\ln h |^2  \leq C_4$ for $h$ sufficiently small and $k \geq 3$, we get
$\|\Phi( v_h)- P_h u\|_{W^{1,\infty}} \leq 2 C_4  h^{k } |\ln h|$.
By triangular inequality and \eqref{2grid-proj1}, the result follows.
%That is $\Phi$ maps $B_h$ into itself.

\end{proof}

We will use below a certain algebraic manipulation which is encoded in the following lemma

\begin{lemma} \label{algebraic}

Let $L_1$ and  $L_2$ be linear functionals and let $L$ denote their product, i.e. $L(v) = L_1(v) L_2(v)$. We have
$$
L(w-u) - L(v-u) = L_1(w-v)L_2(w-u) + L_1(v-u) L_2 (w-v).
$$

\end{lemma}

\begin{proof}
We have using the linearity of $L_1$ and $L_2$
\begin{align*}
L(w-u) - L(v-u) & =  L_1(w-u)L_2(w-u) - L_1(v-u) L_2 (v-u) \\
& = [L_1(w-v) + L_1(v-u)] L_2(w-u) - L_1(v-u) L_2 (v-u) \\
& =  L_1(w-v)L_2(w-u) +  L_1(v-u) [ L_2(w-u) -  L_2 (v-u) ],
\end{align*}
from which the result follows.

\end{proof}

\begin{lemma} \label{contraction}
The mapping $\Phi$ is a strict contraction in $B_h$ for $h$ sufficiently small and $k \geq 2$.
\end{lemma}

\begin{proof}
For $v_h$ and $w_h$  in $B_h$, %$A(u;\Phi(v_h)-P_h u,\phi) = A(u;\Phi(v_h)- u,\phi)$.
we have
\begin{align*}
\begin{split}
A'(u;\Phi( v_h)- \Phi( w_h),\phi) %& = A(u;\Phi(v_h)- u,\phi) \\
&= A'(u;\Phi( v_h)-  v_h,\phi)+A'(u; v_h- w_h,\phi) \\
& \qquad \qquad \qquad \qquad \qquad  + A'(u;w_h - \Phi( w_h),\phi)\\
&= A_{}( w_h,\phi)  -A_{}( v_h,\phi) + A'(u; v_h- w_h,\phi) \\
&=  A_{}( w_h,\phi)  -A_{}( v_h,\phi) + A'(u; v_h -u,\phi) + A'(u; u - w_h,\phi).
\end{split}
\end{align*}
Since $A(u,\phi)=0$, by definition of the residual \eqref{rdef0}, we have
$$
A'(u;\Phi( v_h)- \Phi( w_h),\phi) = R(u;w_h-u,\phi) - R(u;v_h-u,\phi).
$$
Using algebraic manipulations of the type identified in Lemma \ref{algebraic} and \eqref{R1-simple}, we obtain

\begin{multline} \label{expr-con}
A'(u;\Phi( v_h)- \Phi( w_h),\phi) = \frac{1}{2} \sum_{K \in \mathcal{T}_h}  \int_K [(\cof D^2 (w_h - v_h) ) D (w_h -  u)  ]  \cdot D \phi \ud x \\
+ \frac{1}{2} \sum_{K \in \mathcal{T}_h}  \int_K [(\cof D^2 (v_h - u) ) D (w_h -  v_h)  ]  \cdot D \phi \ud x \\
-\frac{1}{2}  \sum_{e \in \mathcal{E}_h^i } \int_e \bigg( \[(\cof D^2 (w_h-v_h)) \] \{\!\{ D (w_h-u) \}\!\}  
+ \[(\cof D^2 (v_h-u)) \] \{\!\{ D (w_h-v_h) \}\!\} \bigg) \phi \ud s \\
+ \frac{1}{2} \sum_{e \in \mathcal{E}_h^i } \int_e \bigg( \[ \{\!\{\cof D^2 (w_h-v_h) \}\!\} D (w_h-u) \] 
+ \[ \{\!\{\cof D^2 (v_h-u) \}\!\} D (w_h-v_h) \] \bigg) \phi\ud s.
\end{multline} 
Arguing as in the proof of Lemma \ref{invariant} and using \eqref{w11-phi}, we obtain
\begin{multline*}
|A'(u;\Phi( v_h)- \Phi( w_h),\phi) | \leq 
 C \Big(   || w_h- v_h||_{W^{2,\infty}(\mathcal{T}_h) } ||w_h-u||_{W^{1,\infty}} \\
+  ||v_h-u||_{W^{2,\infty}(\mathcal{T}_h)}  ||v_h-w_h ||_{W^{1,\infty}} 
   \Big) ||\phi ||_{W^{1,1}}.
\end{multline*}    
As in the proof of  Lemma \ref{invariant}, we have $||v_h-u||_{W^{2,\infty}(\mathcal{T}_h)} \leq C h^{k-1} |\ln h|$ and 
recall that $||w_h-u||_{W^{1,\infty}} \leq C h^{k} |\ln h|$ %and $||v_h-u||_{W^{2,\infty}(\mathcal{T}_h)} \leq C h^{k-1} |\ln h|$ 
by definition of $B_h$. Moreover, by an inverse estimate $|| w_h- v_h||_{W^{2,\infty}(\mathcal{T}_h) } \leq C h^{-1} || w_h- v_h||_{W^{1,\infty} }$. We conclude that
$$
|A'(u;\Phi( v_h)- \Phi( w_h),\phi) | \leq  C\big(        h^{k-1}  |\ln h|+h^{k-1}  |\ln h| \big)||v_h-w_h ||_{W^{1,\infty}}||\phi ||_{W^{1,1}}.
$$
Taking $\phi=g^z_{h,i}$ with the estimate \eqref{2grid-W11greenbound}, and taking $\phi=G^z_h$ with the estimate \eqref{green-0}, we obtain
$$
\|\Phi( v_h)- \Phi( w_h) \|_{W^{1,\infty}} \leq  C ( h^{k-1}+h^{k-1}   )|\ln h|^2  ||v_h-w_h ||_{W^{1,\infty}}, 
$$
that is, for $k \geq 2$ and $h$ sufficiently small, we have $ \|\Phi( v_h)- \Phi( w_h) \|_{W^{1,\infty}} \leq 1/2 ||v_h-w_h ||_{W^{1,\infty}}$.
\end{proof}

The following theorem follows from Lemmas \ref{contraction} and \ref{invariant} and the Banach fixed point theorem.

\begin{thm} \label{2grid-imp-thm}
Problem \eqref{IP-abstract} has a unique solution $u_h$ in $B_h$ for $k \geq 3$, $h$ sufficiently small and
$$||u-u_h||_{W^{1,\infty}} \leq C h^{k} |\ln h |.$$ % 0<\epsilon<1.$$
\end{thm}
We note that in the case of a homogeneous boundary condition, $G=0$ and the right hand side of \eqref{2grid-proj-def} vanishes. In that case, the right hand side of \eqref{expr-inv} simplifies and the rate of convergence in the $W^{1,\infty}$ norm can be shown to be optimal. In other words, Theorem \ref{2grid-imp-thm} can be improved with suitable estimates of the Ritz projection with a non homogeneous boundary condition. 
%While the above theorem gives a suboptimal convergence rate in the $W^{1,\infty}$ norm, it mainly gives existence and uniqueness to a solution of \eqref{IP-abstract}. With this we can now state an 
The following optimal error estimate in the $H^1$ norm is derived from Theorem \ref{2grid-imp-thm}. A different proof was given in \cite{Brenner2010b}. 

\begin{thm}  \label{2grid-imp2-thm}
Problem \eqref{IP-abstract} has a unique solution $u_h$ in $B_h$ for $k \geq 3$, $h$ sufficiently small and
$$||u-u_h||_{H^{1}} \leq C h^{k}. $$
\end{thm}

\begin{proof} The proof is based on the expression \eqref{expr-con} of $A'(u;\Phi( v_h)- \Phi( w_h),\phi)$ derived in the proof of Lemma \ref{contraction} and the expression \eqref{expr-inv} of $A'(u; \Phi( v_h)- P_h u,\phi )$ derived in the proof of Lemma \ref{invariant}. 
Since $\Phi(u_h) = u_h$, we have
\begin{align} \label{u-p-01}
\begin{split}
A'(u; u_h - P_h u,\phi) & = A'(u; \Phi(u_h) - P_h u,\phi) \\
& =  A'(u; \Phi(u_h) - \Phi(P_h u),\phi)  +  A'(u; \Phi(P_h u) - P_h u,\phi). 
\end{split}
\end{align}
In view of \eqref{expr-con}, we obtain %taking $\phi = u_h - P_h u$, 
\begin{multline} \label{u-p-02}
| A'(u; \Phi(u_h) - \Phi(P_h u),\phi)| \leq C  ||u_h-P_h u ||_{W^{2,\infty}(\mathcal{T}_h)}  ||P_h u-u ||_{H^{1}} ||\phi||_{H^1} + \\
C ||u_h- u ||_{W^{2,\infty}(\mathcal{T}_h)}  || u_h- P_h u ||_{H^{1}} ||\phi||_{H^1} + C  ||u_h-P_h u ||_{W^{2,\infty}(\mathcal{T}_h)} 
\sum_{K \in \mathcal{T}_h } | P_h u - u |_{H^{1}(\partial K)} ||\phi||_{L^2(\partial K)} \\
+  C  ||u_h- u ||_{W^{2,\infty}(\mathcal{T}_h)} 
\sum_{K \in \mathcal{T}_h } || P_h u - u_h ||_{H^{1}(\partial K)} ||\phi||_{L^2(\partial K)}.
\end{multline}
By an inverse estimate, Theorem \ref{2grid-imp-thm}, triangle inequality and \eqref{2grid-proj1}, we have $||u_h-P_h u ||_{W^{2,\infty}(\mathcal{T}_h)} \leq
C h^{k-1} | \ln h| $. Similarly, using \eqref{2grid-proj2}, we have $||u_h- u ||_{W^{2,\infty}(\mathcal{T}_h)} \leq C h^{k-1} | \ln h|$.  
Next, by the scaled trace inverse inequality \eqref{trace-inverse} %the trace inequality \eqref{trace-estimate} 
and Cauchy-Schwarz inequality
\begin{multline} \label{u-p-03}
 \sum_{K \in \mathcal{T}_h } || P_h u - u_h ||_{H^{1}(\partial K)} ||\phi||_{L^2(\partial K)}  \leq C h^{-1} 
  \sum_{K \in \mathcal{T}_h } || P_h u - u_h ||_{H^{1}( K)} ||\phi||_{H^1( K)} \\
    \leq C h^{-1}  || P_h u - u_h ||_{H^{1}} ||\phi||_{H^1}.
\end{multline}
%On the other hand, by a standard scaling argument and the Bramble-Hilbert lemma, $| I_h u - u |_{H^{1}(\partial K)} \leq C h^{k-1/2}  ||u||_{H^{k+1}(K)}$. Thus by an inverse estimate

By \eqref{bramble}, an inverse estimate and \eqref{local-nodal-interpolant} %and $| P_h u -  u |_{H^{1}(K)} \leq C h^k  ||u||_{H^{k+1}(K)}$
\begin{align*}
| P_h u - u |_{H^{1}(\partial K)}  & \leq | P_h u - I_h u |_{H^{1}(\partial K)} + | I_h u - u |_{H^{1}(\partial K)}  \\
& \leq C h^{-\frac{1}{2}}  | P_h u - I_h u |_{H^{1}(K)} + C h^{k-\frac{1}{2}} ||u||_{H^{k+1}(K)} \\
& \leq  C h^{-\frac{1}{2}}  | P_h u -  u |_{H^{1}(K)} +  C h^{-\frac{1}{2}}  | u - I_h u |_{H^{1}(K)} + C h^{k-\frac{1}{2}} ||u||_{H^{k+1}(K)}\\
& \leq C h^{-\frac{1}{2}}  | P_h u -  u |_{H^{1}(K)}  + C h^{k-\frac{1}{2}} ||u||_{H^{k+1}(K)}.
\end{align*}
We have $|| P_h u - u ||_{H^{1}} \leq C h^k ||u||_{H^{k+1}}$. This follows from \cite[Theorem 5.4.4]{Brenner08} in the case of homogeneous boundary conditions. The proof of the general case is similar to \eqref{2grid-proj1}.
Thus
$$
\bigg( \sum_{K \in \mathcal{T}_h} | P_h u - u |_{H^{1}(\partial K)} ^2 \bigg)^{\frac{1}{2}} \leq 
C h^{-\frac{1}{2}}  | P_h u -  u |_{H^{1}}  + C h^{k-\frac{1}{2}} ||u||_{H^{k+1}} \leq C h^{k-\frac{1}{2}} ||u||_{H^{k+1}}.
$$ 
As with \eqref{u-p-03}, by \eqref{trace-inverse} and Cauchy-Schwarz inequality, it follows that
\begin{multline} \label{u-p-04}
 \sum_{K \in \mathcal{T}_h } | P_h u - u |_{H^{1}(\partial K)} ||\phi||_{L^2(\partial K)} \leq 
 C h^{-\frac{1}{2} } \sum_{K \in \mathcal{T}_h } | P_h u - u |_{H^{1}(\partial K)} ||\phi||_{H^1( K)} \\
% \leq C h^{k-1}  \sum_{K \in \mathcal{T}_h } ||u||_{H^{k+1}(K)}  ||\phi||_{H^1( K)} \\
  \leq 
  C h^{k-1} ||u||_{H^{k+1} } ||\phi||_{H^1}. 
\end{multline}
We conclude from \eqref{u-p-02}--\eqref{u-p-04} that
\begin{multline*} %
| A'(u; \Phi(u_h) - \Phi(P_h u),\phi)| \leq C h^{k-1} | \ln h| \, || P_h u - u ||_{H^{1}} ||\phi||_{H^1} \\
+ C h^{k-1} | \ln h| \, ||  u_h - P_h u  ||_{H^{1}} ||\phi||_{H^1} 
+  C h^{k-1}  ||u_h - P_h u ||_{W^{2,\infty}} ||\phi||_{H^1}  \\
+ C h^{k-2} | \ln h| \, ||  u_h - P_h u  ||_{H^{1}} ||\phi||_{H^1}.
\end{multline*}
Therefore since by Theorem \ref{2grid-imp-thm} we have the suboptimal estimate $ ||  u_h - P_h u  ||_{H^{1}} \leq C ||u_h - P_h u ||_{W^{1,\infty}}  \leq C h^k | \ln h| $
\begin{multline} \label{u-p-05}
| A'(u; \Phi(u_h) - \Phi(P_h u),\phi)| \leq  C h^{k-1} | \ln h|   h^k ||\phi||_{H^1}  + C h^{k-1} | \ln h|^2   h^k ||\phi||_{H^1} \\
 + C h^{k-2} | \ln h| h^k ||\phi||_{H^1} + C h^{k-2} | \ln h|^2 h^k ||\phi||_{H^1}\\
 \leq C h^{k-2} | \ln h|^2 h^k ||\phi||_{H^1} \leq C h^k ||\phi||_{H^1},
\end{multline}
for $k \geq 3$. By \eqref{expr-inv}, we have
$$
 A'(u; \Phi(P_h u) - P_h u,\phi) = -R(u;P_h u-u,\phi) + A'(u;I_h G - G, \phi).
$$
Using \eqref{R1-simple} and inverse estimates as for \eqref{u-p-04}, we obtain
\begin{multline*} 
| A'(u; \Phi(P_h u) - P_h u,\phi) | \leq C ||P_h u - u ||_{W^{2,\infty}(\mathcal{T}_h)} ||P_h u - u ||_{H^{1}}  ||\phi||_{H^1} \\
+ C  h^{-1} ||P_h u - u ||_{W^{2,\infty}(\mathcal{T}_h)}  || P_h u - u ||_{H^{1}} ||\phi||_{H^1} 
+C\|I_hG-G\|_{H^1}\|\phi\|_{H^1},
\end{multline*}
i.e.
\begin{multline} \label{u-p-06}
| A'(u; \Phi(P_h u) - P_h u,\phi) | \leq C ( h^{k-1} |\ln h| h^k + h^{k-2} |\ln h| h^k + h^k ) ||\phi||_{H^1} \\ \leq C h^k ||\phi||_{H^1}, 
\end{multline}
for $k \geq 3$. Taking $\phi = u_h - P_h u $ in \eqref{u-p-01} and using \eqref{u-p-05} and \eqref{u-p-06}, we get from Poincar\'e's inequality
and $k \geq 3$, for $h$ sufficiently small
$$
 ||  u_h - P_h u  ||_{H^{1}} \leq Ch^k.
$$
This completes the proof by a triangle inequality.

\end{proof}

\section{Analysis of the two-grid algorithm} \label{two-grid-section}

The two-grid discretization for solving nonlinear problems is a well established technique. The nonlinear problem \eqref{IP-abstract} is first solved on a coarse mesh of size $H$. 
The solution $u_H$ is used as an initial guess for one step of Newton's method on the finer mesh of size $h$. 
Both steps are inexpensive and the method is more efficient than solving the problem through multiple iterations of Newton's method directly on the fine mesh. %, provided $H=\text{O}(h^{\lambda})$ with $0<\lambda<1$. 

Since $u$ is smooth and strictly convex, the smallest eigenvalue of $D^2 u$ is uniformly bounded from below. Thus by the continuity of the eigenvalues of the Hessian as a function of its entries and by approximation, $D^2 u_H$ is uniformly positive definite on each element
 for $H$ sufficiently small. A detailed argument was given in \cite[Lemma 4]{AwanouPseudo10} in the context of $C^1$ approximations. We consider the version of \cite[Algorithm 5.5]{Xu1996}. 

{\bf Two-grid algorithm}

\begin{enumerate} \label{2grid-alg1}
\item[1.]
%Given initial guess $u_H^0$: \\
find $u_H \in V_H$, $u_H=g_H$ on $\partial \Omega$, and $A(u_H,\chi) = 0, ~ \forall ~\chi \in  V_H\cap H^1_0(\Omega)$,
\item[2.]  find $u^h \in V_h$, $u^h=g_h$ on $\partial \Omega$, and $A'(u_H;u^h-u_H,\phi)=-A(u_H,\phi),  ~\forall ~\phi \in  V_h\cap H^1_0(\Omega)$.
\end{enumerate}

%%%%%%%%%%%%%%%%%
%We note that a second iteration can be performed on the fine grid, improving accuracy for some choices of $H$ while remaining more efficient than fully carrying out Newton's method on the fine grid. 
Our goal is to show that the two-grid method is optimal in the sense that $||u-u^h||_{H^{1}} \leq C h^k$. %For this we prove that $||u^h-u_h||_{H^{1}} \leq C h^k$ provided $H=h^{\lambda}, 1> \lambda> k/(2k-2- \epsilon)$.
%Our goal is to show convergence of the two-grid method for elements of degree $k \geq 2$.
\begin{thm}
We have the estimate
\begin{equation}\label{2grid-conv-lemma}
||u^h-u_h||_{H^{1}} \leq C h^k,
\end{equation}
for $k \geq 3, H=h^{\lambda}, 1> \lambda>  1/2 + (2+\epsilon)/(2k), 0< \epsilon <1$ and $h$ sufficiently small.
\end{thm}
\begin{proof}
By definition of the two-grid algorithm, the definition of the residual \eqref{rdef0}, and $A(u_h,\phi)=0$ for $\phi \in V_h\cap  W^{1,\infty}_0(\Omega)$, we have

\begin{align*}
A'(u_H; u_h-u^h, \phi)&= A'(u_H; u_h-u_H, \phi) + A'(u_H; u_H-u^h, \phi) \\
& = A'(u_H; u_h-u_H, \phi) + A(u_H,\phi) \\
& = A(u_h,\phi) - R(u_H;u_h-u_H,\phi) = - R(u_H;u_h-u_H,\phi). 
\end{align*}
%Thus, using  \eqref{R1-simple} and \eqref{w11-phi} we have
%$$
%| A'(u_H; u_h-u^h, \phi) | \leq | R(u_H;u_h-u_H,\phi)| \leq C ||u_h-u_H ||_{W^{2,\infty}(\mathcal{T}_h)} ||u_h-u_H ||_{W^{1,\infty}}  ||\phi ||_{W^{1,1}}. 
%$$
It follows that
\begin{align*}
A'(u; u_h-u^h, \phi)& = A'(u-u_H; u_h-u^h, \phi) + A'(u_H; u_h-u^h, \phi) \\
& = A'(u-u_H; u_h-u^h, \phi) - R(u_H;u_h-u_H,\phi). 
\end{align*}
With arguments similar to the ones used in the proof of Theorem \ref{2grid-imp2-thm}, we have
\begin{multline} \label{u-p-066}
| A'(u; u_h-u^h, \phi) | \leq C  || u - u_H ||_{W^{2,\infty}(\mathcal{T}_h)}  || u_h - u^h ||_{H^{1}} ||\phi||_{H^1} \\
+ C h^{-1} || u - u_H ||_{W^{2,\infty}(\mathcal{T}_h)}  || u_h - u^h ||_{H^{1}} ||\phi||_{H^1} 
+ C || \phi ||_{L^{\infty}} \sum_{K \in \mathcal{T}_h } | u - u_H |_{H^{1}(\partial K)} ||  u_h - u^h ||_{H^2(\partial K)}\\
+ |R(u_H;u_h-u_H,\phi)|.
\end{multline}
But, using \eqref{bramble}, \eqref{local-nodal-interpolant} and an inverse estimate
\begin{align*}
| u - u_H |_{H^{1}(\partial K)} & \leq | u - I_H u |_{H^{1}(\partial K)} + | I_H u - u_H |_{H^{1}(\partial K)} \\
& \leq C H^{k-\frac{1}{2}} ||u||_{H^{k+1}(K)} + C H^{-\frac{1}{2}}  | I_H u - u_H |_{H^{1}( K)} \\
& \leq C H^{k-\frac{1}{2}} ||u||_{H^{k+1}(K)} + C H^{-\frac{1}{2}}  | I_H u - u |_{H^{1}( K)} + C H^{-\frac{1}{2}}  | u - u_H |_{H^{1} ( K)}. 
%& \leq C H^{k-\frac{1}{2}} ||u||_{H^{k+1}(K)}. 
\end{align*}
It follows that
$$
\bigg( \sum_{K \in \mathcal{T}_h } | u - u_H |^2_{H^{1}(\partial K)} \bigg)^{\frac{1}{2}} \leq C H^{k-\frac{1}{2}} ||u||_{H^{k+1}} + C H^{-\frac{1}{2}}  | I_H u - u |_{H^{1}} + C H^{-\frac{1}{2}}  | u - u_H |_{H^{1} }.
$$
We therefore obtain from \eqref{nodal-interpolant} and Theorem \ref{2grid-imp2-thm}
$$
\bigg( \sum_{K \in \mathcal{T}_h } | u - u_H |^2_{H^{1}(\partial K)} \bigg)^{\frac{1}{2}} \leq C H^{k-\frac{1}{2}} ||u||_{H^{k+1}}. 
$$
By Cauchy-Schwarz's inequality, \eqref{trace-inverse} and an inverse estimate, it follows that
$$
\sum_{K \in \mathcal{T}_h } | u - u_H |_{H^{1}(\partial K)} ||  u_h - u^h ||_{H^2(\partial K)} \leq C h^{- \frac{3}{2}}  H^{k-\frac{1}{2}} ||u||_{H^{k+1}}
 || u_h - u^h ||_{H^{1}}. 
$$
Therefore, since $|| u - u_H ||_{W^{2,\infty}(\mathcal{T}_h)} \leq C H^{k-1} |\ln H|$ and by the discrete Sobolev inequality, c.f.  \cite{Bramble86},
$|| \phi ||_{L^{\infty}}  \leq C (1+ | \ln h|^{ 1/2 }) ||\phi||_{H^1}$, we obtain from \eqref{u-p-066} 
\begin{multline} \label{u-p-07}
| A'(u; u_h-u^h, \phi) | \leq C h^{-1} || u - u_H ||_{W^{2,\infty}(\mathcal{T}_h)}  || u_h - u^h ||_{H^{1}} ||\phi||_{H^1} \\
+ C H^{k-\frac{1}{2}} h^{- \frac{3}{2}}   | \ln h| \, || u_h - u^h ||_{H^{1}}  ||u||_{H^{k+1} } ||\phi||_{H^1} + |R(u_H;u_h-u_H,\phi)| \\
\leq C h^{-1} H^{k-1} |\ln H| \, || u_h - u^h ||_{H^{1}} ||\phi||_{H^1} 
+ C H^{k-\frac{1}{2}} h^{- \frac{3}{2}} |\ln h| \,   || u_h - u^h ||_{H^{1}} ||\phi||_{H^1} \\
+ |R(u_H;u_h-u_H,\phi)|. 
\end{multline}
Using \eqref{R1-simple} and %arguing as in the proof of Lemma \ref{2grid-imp2-thm}, 
the trace estimates \eqref{trace-inverse} and \eqref{trace-estimate}, we have
\begin{multline} \label{u-p-08}
 |R(u_H;u_h-u_H,\phi)| \leq C  || u_h - u_H ||_{W^{2,\infty}(\mathcal{T}_h)}  || u_h - u_H ||_{H^{1}} ||\phi||_{H^1} \\
+ C h^{-1} || u_h - u_H ||_{W^{2,\infty}(\mathcal{T}_h)}  || u_h - u_H ||_{H^{1}} ||\phi||_{H^1}  \\
\leq C h^{-2} || u_h - u_H ||_{W^{1,\infty}}  || u_h - u_H ||_{H^{1}} ||\phi||_{H^1}. 
\end{multline}
Taking $\phi = u_h - u^h $ in \eqref{u-p-07} and using \eqref{u-p-08}, we get from Poincar\'e's inequality, Theorems \ref{2grid-imp-thm} and \ref{2grid-imp2-thm}
for $k \geq 3$ and $h$ sufficiently small
\begin{multline*}
|| u_h - u^h ||_{H^{1}} \leq C (h^{-1} H^{k-1} |\ln H| + H^{k-\frac{1}{2}} h^{- \frac{3}{2}} |\ln h|  ) || u_h - u^h ||_{H^{1}} \\
+C h^{-2} (h^{k} |\ln h | + H^{k} |\ln H |) (h^k+H^k).
\end{multline*}
We conclude that for $H=h^{\lambda}$, $\lambda > \max \{1/(k-1), 3/(2k-1) \} = 3/(2k-1)$ and $k \geq 3$, 
$$
|| u_h - u^h ||_{H^{1}} \leq C h^{-2}  |\ln H \, | H^{2k}.
$$
We therefore get $||u^h-u_h||_{H^{1}} \leq C h^k$ provided $\lambda >  3/(2k-1)$ and $2 \lambda k -2- \epsilon > k$ for some $\epsilon \in (0,1) $, that is 
$\lambda >  \max \{ (k+2+ \epsilon)/(2k), 3/(2k-1) \} = 1/2 + (2+\epsilon)/(2k)$ for $k \geq 3$. 
\end{proof}

\section{Numerical experiments}\label{num}

The computational domain is taken to be the unit square $[0,1]^2$. 
A uniform grid is obtained by dividing the domain into smaller equal size squares, then dividing each square into two triangles by taking the diagonal with positive slope. 
We consider a smooth convex test function $u(x,y)=e^{(x^2+y^2)/2}$ so that 
$f(x,y)=(1+x^2+y^2)e^{(x^2+y^2)}$ and $g(x,y)=e^{(x^2+y^2)/2}$ on $\partial \Omega$. 
While our convergence analysis is only for cubic and higher order elements, we believe the results should be true for quadratic elements. We show numerical convergence of the two-grid solution $u^h$ to the continuous solution $u$ using $P_2$ finite elements.
%\begin{figure}[tbp]\begin{center}
%\includegraphics[angle=0, height=5cm]{??.png}
%caption{Initial unstructured triangulation} \label{fig1} \end{center} \end{figure}
On the coarse grid of size $H$, we first seek an initial guess $u_H^0$ of $u_H$ as the standard finite element approximation of the solution $u_0$ of
\begin{equation*}
\Delta u_0=2\sqrt{f}, \quad u_0=g \text{ on } \partial \Omega.
\end{equation*}
% the Poisson finite element problem
%\begin{equation*}
%\int_{\Omega}Du_0\cdot Dv~ dx= -2\int_{\Omega} \sqrt{f}~v~dx, \quad u_0=g_h \text{ on } \partial \Omega. 
%\end{equation*}
%cite?
%We consider a fine grid of size $h=2^{-n}$, where $n\geq 2$. 
For solving the coarse grid problem, we perform Newton's method on the coarse grid, setting the maximum iterations to 10 and we impose that the algorithm terminates when 
%$\|u_h^{j+1}-u_h^j\|_{L^2} < 10^{-12}$ for some $j$.
%$\frac{\|u_H\|_{L^{\infty}}}{\|u^0_H\|_{L^{\infty}}}\leq 10^{-6}.$
$\|u_H\|_{L^{\infty}}/ \|u^0_H\|_{L^{\infty}} \leq 10^{-6}.$
We report computation times (in seconds) for the two-grid method and Newton's method on the fine grid, as well as  $H^1$ errors and associated rate of convergence. 
See Table \ref{tab1} for $\lambda=1+ 2 \ln 2/(\ln h) =1-2/n, h=1/2^n, n=2, 3 \ldots$, and Table \ref{tab2} for $\lambda=1+  \ln 2/(\ln h)=1-1/n, h=1/2^n, n=2, 3 \ldots$
%We consider two ways of choosing the coarse grid parameter $H$ such that $H=h^{\lambda}$. 
%We first choose $H$ such that for grid size $h=2^{-n}$, where $n=2,3,\dots$ we then take $H=2^{-n+1}$. Then for $h<1/2$, this gives $\lambda_1=\frac{\log 2}{\log h}+1$. For the second choice of $H$ we simply take $\lambda_2=1/2$. The finest grid we consider is $h=2^{-8}$. 

%For the numerical results we present here, we chose $\lambda$ such that for grid size $h=2^{-n}$ where $n=2,3,\dots$, we then take...
%Then for $h<1/2$, this gives $\lambda=\frac{\log 2}{\log h}+1$.

Experiments with $P_3$ elements gave the expected results, and so we do not report those here.
We also attempted several multigrid experiments, where we interpolate between a series of meshes before ending on the fine grid. However these gave results comparable to the two-grid algorithm. At the cost of extra computation time, there is a slight increase in accuracy if a second iteration is performed on the fine grid. We do not report these results since a second iteration on the fine grid does not affect the rate of convergence of the method. 

The two-grid computations are accurate and fast compared with Newton's method.
 %where the errors and convergence rates given are from the results on the fine grid. 
The computations were done in FreeFEM++ on an HP computer with Pentium dual-core 2.60 GHz processor running Windows 10. 
%
%%2 iterations
%\begin{table} 
%\begin{tabular}{|c|c|c|c|c|c|c|c|} \hline
%$H$ & $h$ &$||u-u_h||_{H^1}$ & rate & $||u-u^h||_{H^1}$ & rate & two-grid time & Newton time\\ \hline
%$1/2^0$ & $1/2^2$ &2.19 $10^{-2}$ & - & 1.18 $10^{-1}$ & -             & 3.90 $10^{-2}$& 4.90 $10^{-2}$ \\ \hline
%$1/2^1$ &$1/2^3$ & 5.55 $10^{-3}$ & 1.98 & 1.61 $10^{-2}$ & 2.88 & 1.15 $10^{-1}$ & 1.67 $10^{-1}$ \\ \hline
%$1/2^2$ &$1/2^4$ &1.39 $10^{-3}$ &2.00 & 1.41 $10^{-3}$ & 3.52 & 4.15 $10^{-1}$ &8.58 $10^{-1}$\\ \hline
%$1/2^3$ & $1/2^5$ &3.48 $10^{-4}$ & 2.00 & 3.47 $10^{-4}$ & 2.02 & 1.79 $10^{0}$&3.25 $10^{0}$\\ \hline
%$1/2^4$ & $1/2^6$ & 8.70 $10^{-5}$ & 2.00 &8.69 $10^{-5}$ & 2.00 & 7.29 $10^{0}$ & 1.31 $10^{1}$\\ \hline
%$1/2^5$ & $1/2^7$ & 2.18 $10^{-5}$ &2.00 & 2.17 $10^{-5}$ & 2.00 & 3.30 $10^{1}$& 5.59 $10^{1}$\\ \hline
%$1/2^6$ & $1/2^8$ & 5.44 $10^{-6}$ & 2.00 & 5.44 $10^{-6}$ & 2.00 &1.57 $10^{2}$& 2.73 $10^{2}$\\ \hline
%%$1/2^7$ & $1/2^9$ & N/A & N/A & 2.24$10^{-3}$ & 1.03 &3.60 $10^{4}$ & N/A\\ \hline
%%7 &1.8956 $10^{-6}$ & 1.7983 $10^{-5}$ &4.8870 $10^{-4}$ \\ \hline
%%8 & & & \\ \hline
%%Rate & $10^{2.3278}$ $d^{-9.4276}$& $10^{3.0318}$ $d^{-9.0347}$& $10^{3.4794}$ $d^{-7.8305}$ \\ \hline
%\end{tabular}
%\caption{ $\lambda=1+ \frac{2 \ln 2}{\ln h} $} \label{tab1}
%\end{table}

%1 iteration
\begin{table} 
\begin{tabular}{|c|c|c|c|c|c|c|c|} \hline
$H$ & $h$ &$||u-u_h||_{H^1}$ & rate & $||u-u^h||_{H^1}$ & rate & two-grid time & Newton time\\ \hline
$1/2^0$ & $1/2^2$ &2.19 $10^{-2}$ & -       & 3.12 $10^{-1}$ & -             & 2.80 $10^{-2}$& 4.90 $10^{-2}$ \\ \hline
$1/2^1$ &$1/2^3$ & 5.55 $10^{-3}$ & 1.98 & 3.22 $10^{-2}$ & 3.28 & 7.10 $10^{-2}$ & 1.67 $10^{-1}$ \\ \hline
$1/2^2$ &$1/2^4$ &1.39 $10^{-3}$ &2.00   & 3.69 $10^{-3}$ & 3.12 & 2.37 $10^{-1}$ &8.58 $10^{-1}$\\ \hline
$1/2^3$ & $1/2^5$ &3.48 $10^{-4}$ & 2.00 & 6.34 $10^{-4}$ & 2.54 & 9.92 $10^{-1}$&3.25 $10^{0}$\\ \hline
$1/2^4$ & $1/2^6$ & 8.70 $10^{-5}$ & 2.00& 1.48 $10^{-4}$ & 2.10 & 4.15 $10^{0}$ & 1.31 $10^{1}$\\ \hline
$1/2^5$ & $1/2^7$ & 2.18 $10^{-5}$ &2.00 & 3.65 $10^{-5}$ & 2.02 & 1.87 $10^{1}$& 5.59 $10^{1}$\\ \hline
$1/2^6$ & $1/2^8$ & 5.44 $10^{-6}$ & 2.00 & 9.11 $10^{-6}$ & 2.00 &8.57 $10^{1}$& 2.73 $10^{2}$\\ \hline
%$1/2^7$ & $1/2^9$ & N/A & N/A & 2.24$10^{-3}$ & 1.03 &3.60 $10^{4}$ & N/A\\ \hline
%7 &1.8956 $10^{-6}$ & 1.7983 $10^{-5}$ &4.8870 $10^{-4}$ \\ \hline
%8 & & & \\ \hline
%Rate & $10^{2.3278}$ $d^{-9.4276}$& $10^{3.0318}$ $d^{-9.0347}$& $10^{3.4794}$ $d^{-7.8305}$ \\ \hline
\end{tabular}
\caption{ $\lambda=1+ \frac{2 \ln 2}{\ln h} $} \label{tab1}
\end{table}

%1 iteration
\begin{table} 
\begin{tabular}{|c|c|c|c|c|c|c|c|} \hline
$H$ & $h$ &$||u-u_h||_{H^1}$ & rate & $||u-u^h||_{H^1}$ & rate & two-grid time & Newton time\\ \hline
$1/2^1$ & $1/2^2$ &2.19 $10^{-2}$ & -      & 2.72 $10^{-2}$       & -       & 3.10 $10^{-2}$& 4.90 $10^{-2}$ \\ \hline
$1/2^2$ &$1/2^3$ & 5.55 $10^{-3}$ & 1.98 & 5.97 $10^{-3}$ & 2.19 & 7.80 $10^{-2}$ & 1.67 $10^{-1}$ \\ \hline
$1/2^3$ &$1/2^4$ &1.39 $10^{-3}$ &2.00   & 1.43 $10^{-3}$ & 2.06 & 3.44 $10^{-1}$ &8.58 $10^{-1}$\\ \hline
$1/2^4$ & $1/2^5$ &3.48 $10^{-4}$ & 2.00 & 3.54 $10^{-4}$ & 2.01 & 1.69 $10^{0}$&3.25 $10^{0}$\\ \hline
$1/2^5$ & $1/2^6$ & 8.70 $10^{-5}$ & 2.00 & 8.83 $10^{-5}$ & 2.00 & 6.44 $10^{0}$ & 1.31 $10^{1}$\\ \hline
$1/2^6$ & $1/2^7$ & 2.18 $10^{-5}$ &2.00 & 2.21 $10^{-5}$ & 2.00 & 3.18 $10^{1}$& 5.59 $10^{1}$\\ \hline
$1/2^7$ & $1/2^8$ & 5.44 $10^{-6}$ & 2.00 & 5.51 $10^{-6}$ & 2.00 &1.20 $10^{2}$& 2.73 $10^{2}$\\ \hline
%$1/2^7$ & $1/2^9$ & N/A & N/A & 2.24$10^{-3}$ & 1.03 &3.60 $10^{4}$ & N/A\\ \hline
%7 &1.8956 $10^{-6}$ & 1.7983 $10^{-5}$ &4.8870 $10^{-4}$ \\ \hline
%8 & & & \\ \hline
%Rate & $10^{2.3278}$ $d^{-9.4276}$& $10^{3.0318}$ $d^{-9.0347}$& $10^{3.4794}$ $d^{-7.8305}$ \\ \hline
\end{tabular}
\caption{ $\lambda=1+ \frac{ \ln 2}{\ln h} $} \label{tab2}
\end{table}

%\bibliographystyle{abbrv} 
%\bibliography{/Users/gerard/Desktop/MathHome4/Monge/SplineMonge08}
%\bibliography{bib1thesis}

\end{document}